\documentclass{amsart}
\usepackage{BoixGomezRamirez2023}

\begin{document}

\title[Algebraic and geometric extensions of Goldbach's conjecture]
{On some algebraic and geometric extensions of Goldbach's conjecture}

\author[Danny A. J. G\'omez-Ram\'irez]{Danny A. J. G\'omez-Ram\'irez}
\address{Instituci\'on Universitaria Pascual Bravo, Medell\'in, Colombia. Visi\'on Real Cognitiva (Cognivisi\'on) S.A.S. Itagu\'i, Colombia.}
\email{daj.gomezramirez@gmail.com}

\author[A. F. Boix]{Alberto F. Boix}
\address{IMUVA–Mathematics Research Institute, Universidad de Valladolid, Paseo de Belen, s/n, 47011,
Valladolid, Spain.}
\email{alberto.fernandez.boix@uva.es}



\begin{abstract}
The goal of this paper is to study Goldbach's conjecture for rings of regular functions of affine algebraic varieties over a field. Among our main results, we define the notion of Goldbach condition for Newton polytopes, and we prove in a constructive way that any polynomial in at least two variables over a field can be expressed as sum of at most $2r$ absolutely irreducible polynomials, where $r$ is the number of its non--zero monomials. We also study other weak forms of Goldbach's conjecture for localizations of these rings. Moreover, we prove the validity of Goldbach's conjecture for a particular instance of the so--called forcing algebras introduced by Hochster. Finally, we prove that, for a proper multiplicative closed set $S$ of $\mathbb{Z}$, the collection of elements of $S^{-1}\mathbb{Z}$ that can be written as finite sum of primes forms a dense subset of the real numbers, among other results.
\end{abstract}


\maketitle

\section*{Introduction}
Goldbach´s conjecture has been one of the most famous open problems in elementary Number Theory and in Mathematics, partly because its simple description and for the difficulty of finding a general solution for it. It was originally proposed in a letter exchange between the mathematicians Christian Goldbach and Leonard Euler in 1742 (see \cite[pages 189--191]{GoldbachtoEulerconjecture} for the original letter, for an English translation see \cite{GoldbachtoEulerconjecturetranslated}). In modern terms it states that any even integer number bigger or equal than 4 can be written as the sum of two prime numbers \cite{vaughan2016goldbach}. It is elementary to see that the conjecture is equivalent to show that any integer number bigger or equal to 5 can be written as the sum of at most 3 primes.\footnote{Here, as usual, 1 is not considered a prime number. However, in the original description of Goldbach, 1 was considered a prime.} Another elementary equivalent form states that for any integer $x$ and $n$, both bigger or equal than 2, the number $nx$ can be expressed as the sum of exactly $n$ prime numbers. In other words, there exists prime numbers $p_1,\cdots,p_n$ such that 

\begin{equation}\label{seminal}
 nx=\sum_{i=1}^{n}p_i.
\end{equation}

Although the original conjecture has not been resolved, there has been a lot of outstanding results partial results by J. J. Sylvester, G. H. Hardy, J. E. Littlewood, S. Ramanujan, I. M. Vinogradov, using asymptotic methods, whose combined and improved techniques are known nowadays as the \textit{Hardy-Littlewood-Ramanujan-Vinogradov method}; as well as closed related statements developed over the integers and proved by E. Bombiery, T. Tao and B. J. Green, among many others (for a complete historical and technical review see \cite{vaughan2016goldbach}). 

Regarding extensions of Goldbach's conjecture to other commutative rings with unity, one of the most interesting elementary results is the one proved in 1965 by D. R. Hayes (rediscovered 30 years after by A. Rattan and C. Stewart \cite{RattanStewart}) which proves the conjecture for the ring $\mathbb{Z}[x]$. More precisely, it says that any polynomial of degree $n$ with integer coefficients can be written as the sum of two irreducible polynomials of the same degree \cite[Theorem 1]{hayes1965goldbach}. His proof was completely elementary and used mainly the classic Eisenstein's criterion for irreducibility of polynomials over the integers. In the same paper, Hayes also proved the conjecture for $D[x]$ where $D$ is a principal ideal domain which contains infinitely many prime elements \cite[Theorem 2]{hayes1965goldbach}. Hayes result was also proved by F. Saidak \cite{Saidakintegerpolynomials}, who also provides an upper bound for the number of ways one can write a polynomial with integer coefficients as sum of two irreducible ones with a prescribed upper bound on the coefficients. More than 40 years after, P. Pollack was able to find a kind of generalization of Hayes' argument to prove the corresponding version of Hayes' theorem but for polynomials over a Noetherian integral domain with infinitely many maximal ideals \cite[Theorem 1]{pollack2011polynomial}; or over a ring of polynomials over an integral domain \cite[Theorem 2]{pollack2011polynomial}. In this context, A. Bodin, P. D\`ebes and S. Najib, building upon new results on the Schinzel's hypothesis, proved Goldbach's conjecture for the ring $R[x]$, where $R$ is a unique factorization domain with fraction field satisfying certain technical conditions, the interested reader can consult \cite[Theorem 1.1 and Corollary 1.5]{bodin2020schinzel} for further details.


For rings of formal power series, the situation is a bit more complicated. As shown by E. Paran in \cite[Theorem 1.1]{Goldbachpowerseries}, a formal power series
\[
f=\sum_{i\geq 0}f_i x^i\in\mathbb{Z}[\![x]\!]
\]
is a sum of two irreducible power series if and only if $f_0$ is either of the form $\pm p^k\pm q^l$ or of the form $\pm p^k$, where $p,\ q$ are prime integers and $k,\ l$ are positive integers. In particular, this shows that not all the elements of $\mathbb{Z}[\![x]\!]$ can be written as sum of two irreducible power series, the interested reader may like to consult \cite[page 454]{Goldbachpowerseries} for a concrete example.

On the other hand, we will get heuristic inspiration in the form of using the cognitive (metamathematical) processes called (generic) exemplification in the context of Cognitive-Computational Metamathematics (CCMM) (or Artificial Mathematical Intelligence \cite{AMI})  as a starting point for developing new interesting mathematical structures (e.g. concepts, proofs and theories) \cite[Chapter 10, \S10.3]{AMI}. In other words, an implicit heuristic pillar of our presentation will be to use \eqref{seminal}, (as an initial (generic) exemplification) and slight generalizations of it as a kind of working criterion for studying and generating further algebraic structures in the context of Commutative Algebra and classic Algebraic Geometry, e.g. coordinate rings of affine algebraic varieties over an algebraically closed field. 

To clarify more explicitly what we mean here in terms of exemplification, we refer the more curious reader to \cite[\S 6]{gomezbrennernormality} and \cite{gomez2013homological}, where one appreciates how important a simple concept can be, like the one of forcing algebra and a concrete instance of it, as a conceptual intersection point where other seminal mathematical notions ad-hoc meet. In this way, one can start adopting a paradigm-shifting perspective in doing mathematics, where concrete examples constitute the starting working points to compare and develop new theories and interesting mathematical structures.

Going back to our original question, the surprising work by G. Effinger, K. Hicks and G. L. Mullen, (see \cite{effinger2005integers} and the references given therein), shows in a highly intuitive manner that the integers and the ring of polynomials over a finite field are highly close algebraic structures, perhaps closer than it might seem at first glance. For instance, one can define and extend almost exactly some seminal aspects of classic analytic methods (due originally to Hardy, Littlewood, Ramanujan and Vinogradov) for stating and subsequently proving (contingent) results related to Goldbach and the twin-primes conjectures for the ring $\mathbb{F}_q[x]$, that were originally developed for $\mathbb{Z}$.

This suggests that studying Goldbach conjecture for rings of polynomials in several variables over a (several types of) field(s) could give considerable additional insight for obtaining new techniques for tackling the classic Goldbach conjecture.

So, we can start with a field of coefficients being algebraically rich enough such that one can easily check if Goldbach conjecture (GC) holds (or not) for the corresponding ring of regular functions \cite{gortz2010algebraic}. Thus, in the next section we study GC for special rings of coordinates of affine varieties.

Now, we provide a brief summary of the contents of this paper for the convenience of the reader. In Section \ref{main result section}, we present one of the main results of the paper, which essentially says that any polynomial in at least two variables over a field can be written explicitly as sum of at most $2r$ absolutely irreducible polynomials, where $r$ denotes the number of its non--zero monomials (see \thref{teo1}). Our proof of \thref{teo1} is constructive, which leads us in Section \ref{implemented result section} to present an algorithm, implemented in Macaulay2 \cite{M2}, that given as input a polynomial in at least two variables over a field, returns as output its decomposition in $2r$ absolutely irreducible polynomials. On the other hand, in Section \ref{section on Goldbach conjecture and localization} we study some weak forms of Golbach's conjecture over the integers and its localizations; more precisely, we want to provide some partial positive answers to the non--trivial question of whether Golbach's conjecture holds over $S^{-1}\mathbb{Z}$, where $S$ is a non--trivial multiplicative closed set. In Section \ref{section of more localization}, we study again Golbach's conjecture on the localization of a polynomial ring over a field, showing as main result (see \thref{localization}) that the conclusion of \thref{teo1} still holds replacing the ring of polynomials by a suitable localization of it. In Section \ref{section on forcing algebras}, we obtain as main result (see \thref{Goldbach conjecture and forcing algebras}) that Goldbach's conjecture holds over a very particular case of the so--called forcing algebras, that were introduced by M.\,Hochster in \cite{solidclosure} and later developed specially by H.\,Brenner and D. A. J. G\'omez--Ram\'irez (see \cite{gomezbrennernormality} and the references given therein). In Section \ref{section on strong Golbach condition}, essentially buinding upon Pollack's result, we exhibit some coordinate rings of affine algebraic varieties where any regular function can be written as a sum of two irreducible ones (see \thref{prop2} and its corollaries). Finally, in Section \ref{section on the curve case} we present some examples to illustrate that the validity or not of Golbach's conjecture is more complicated to dealt with for rings attached to affine algebraic curves.

\section{Explicit Weak Forms of the Goldbach Conjecture for some special classes of Polynomial (Coordinate) Rings over Fields and Commutative Rings with Unity}\label{main result section}

As we mentioned in the introduction, we want to study a more general form of \eqref{seminal}, where we allow more flexibility in the parameter $n$. In other words, we will study some special collections of coordinate rings where equations of the form 

\begin{equation}\label{seminal2}
n_1H=\sum_{i=1}^{n_2}p_i
\end{equation}
holds, where $n_1,n_2\in\mathbb{N}$ are parameters that can vary independently or not (from each other as well as from the element $H$).

Let us start with a fundamental (algebraic) structure in classic algebraic geometry: the ring of regular functions on the complex affine space $\mathbb{A}_{\mathbb{C}}^n$, i.e., the ring of polynomials $\mathcal{O}(\mathbb{A}_{\mathbb{C}}^n)=\mathbb{C}[x_1,\ldots,x_n]$. This classic geometric-algebraic space is a suitable starting point, since a lot of further algebraic constructions (in characteristic zero) are implicitly related with it (e.g. the ring of coordinates of any affine complex variety is simply a quotient of the one of  $\mathbb{A}_{\mathbb{C}}^n$). Moreover, it is highly surprising that despite of  $\mathbb{A}_{\mathbb{C}}^n$ being one of the most canonical and primary spaces of study in classic algebraic geometry, a lot of information regarding its geometry and its symmetries is still a mystery \cite{kraft1996challenging}. 
Now, regarding \eqref{seminal2} for the ring $\mathcal{O}(\mathbb{A}_{\mathbb{C}}^n)$, one sees straightforwardly that \thref{prop2} implies that a stronger form of the GC holds for $\mathcal{O}(\mathbb{A}_{\mathbb{C}}^n)$ for $n\geq 2$, i.e., any polynomial can be written as the sum of two irreducible ones (i.e., $n_1=1$ and $n_2=2$ in \eqref{seminal2}). However, the proof of this result (and the necessary lemmas needed) do not give explicit description of the desired decomposition into irreducibles.

In order to do so, we assume that the reader has some basic familiarity with the notion of Newton polytope associated with a polynomial (see for example \cite[Chapter 2, pages 9--12]{SturmfelsGBCPbook}). Anyway, in what follows we recall some notions for the sake of completeness.

A polynomial $f\in K[x_1,\ldots,x_n]$ is \textbf{absolutely irreducible} if it remains irreducible over each algebraic extension of $K.$ More generally, and inspired by \cite{Koyuncuirreducibilityoverrings}, given a commutative ring $R$ and $g\in R[x_1,\ldots,x_n],$ we say that $g$ is \textbf{absolutely irreducible} if it remains irreducible over each ring extension $R\subseteq S.$ We also recall here that, given a multiindex $\mathbf{a}=(a_1,\ldots,a_n)\in\mathbb{N}^n,$ we will denote by $\mathbf{x}^{\mathbf{a}}$ the monomial
\[
\mathbf{x}^{\mathbf{a}}:=\prod_{i=1}^n x_i^{a_i},
\]
and by $\supp (\mathbf{a})$ the set $\supp (\mathbf{a}):=\{i\in [n]:\ a_i\neq 0\}$, where $[n]:=\{1,\ldots,n\}$. Moreover, we denote by $\gcd (\mathbf{a})$ the greatest common divisor of its components; that is $\gcd (\mathbf{a}):=\gcd (a_1,\ldots,a_n).$ Finally, we review the notion of integrally indecomposable polytope.

\begin{definition}\label{integral decomposability: definition}
A point in $\mathbb{R}^n$ is called \textbf{integral} if its coordinates are integers. A polytope in $\mathbb{R}^n$ is called \textbf{integral} if all its vertices are integral.

Moreover, an integral polytope $C$ is called \textbf{integrally decomposable} if there are integral polytopes $A$ and $B$ such that $C=A+B,$ where both $A$ and $B$ have at least two points, and $+$ denotes the Minkowski sum of polytopes. Otherwise, we say that $C$ is \textbf{integrally indecomposable}.

Finally, given two points $\mathbf{p},\ \mathbf{q}\in\mathbb{R}^n$ we denote by $\overline{\mathbf{p}\mathbf{q}}$ the line segment starting at $\mathbf{p}$ and ending at $\mathbf{q}.$
\end{definition}
As pointed out along \cite{GaoLauderalgorithm}, the problem of testing whether a given polytope is integrally indecomposable is NP--complete. Our interest for integrally indecomposable polytopes stems from the following irreducibility criterion obtained by Gao in \cite[page 507]{gao2001absolute}.

\begin{theorem}[Gao irreducibility criterion]\thlabel{Gao irreducibility criterion}
Let $K$ be any field, and let $f\in K[x_1,\ldots,x_n]$ be a non--zero polynomial not divisible by any of the $x_i$'s. If the Newton polytope of $f$ is integrally indecomposable, then $f$ is absolutely irreducible.
\end{theorem}

The extension of Gao irreducibility criterion for certain rings was given by Koyuncu in \cite[Corollary 2.4]{Koyuncuirreducibilityoverrings}.

\begin{theorem}[Koyuncu irreducibility criterion]\thlabel{Koyuncu irreducibility criterion}
Let $R$ be a commutative ring, and $f\in R[x_1,\ldots,x_n]$ a non--zero polynomial not divisible by any of the $x_i$'s. Suppose that the coefficients of all the terms forming the vertices of the Newton polytope $P_f$ of $f$ are non--zero divisors of $R.$ Then, if $P_f$ is integrally indecomposable, then $f$ is absolutely irreducible.
\end{theorem}

In the proof of our first main result (see \thref{teo1}) we plan to use in a crucial way \nameref{Gao irreducibility criterion} combined with the following construction, also due to Gao \cite[Theorem 4.2]{gao2001absolute}.

\begin{proposition}\thlabel{integrally indecomposable pyramid}
Let $Q$ be any integral polytope in $\mathbb{R}^n$ contained in a hyperplane $H$ and let $\mathbf{v}\in\mathbb{R}^n$ be an integral point lying outside of $H.$ Suppose that the vertices $\mathbf{v}_1,\ldots,\mathbf{v}_r$ are all the vertices of $Q.$ Then, the polytope $\conv (Q,\mathbf{v})$ given by the convex hull of $\mathbf{v}_1,\ldots,\mathbf{v}_r,\mathbf{v}$ is integrally indecomposable if and only if
\[
\gcd (\mathbf{v}-\mathbf{v}_1,\ldots,\mathbf{v}-\mathbf{v}_r)=1.
\]
\end{proposition}

Inspired also by Gao and Koyuncu criteria, we also want to introduce here the following definition.

\begin{definition}\label{Goldbach conjecture for polytopes}
Let $n\geq 1$ be an integer, and let $P\subseteq\mathbb{R}^n$ be a polytope of the form $P=\conv (\mathbf{v}_1,\ldots,\mathbf{v}_r)$ for some $\mathbf{v}_1,\ldots,\mathbf{v}_r\in\mathbb{N}^n.$ We say that $P$ satisfies \textbf{Goldbach condition} if either $P$ is integrally indecomposable, or there are points $\mathbf{w}_1,\ldots,\mathbf{w}_s\in\mathbb{N}^n$ such that:

\begin{enumerate}[(i)]

\item $\conv (\mathbf{w}_1,\ldots,\mathbf{w}_s)$ is integrally indecomposable.

\item We have
\[
\bigcap_{i=1}^s \supp (\mathbf{w}_i)=\emptyset.
\]

\item $\conv (\mathbf{v}_1,\ldots,\mathbf{v}_r,\mathbf{w}_1,\ldots,\mathbf{w}_s)$ is integrally indecomposable.

\end{enumerate}
\end{definition}
Our reason for introducing Goldabch condition on polytopes is given by the following result.

\begin{proposition}\thlabel{from polytopes to polynomials}
Let $K$ be any field, and let $f\in K[x_1,\ldots, x_n]$ be a polynomial not divisible by any of the $x_i$'s. Assume that the Newton polytope $N(f)$ of $f$ satisfies Goldbach condition. Then, we have that either $f$ is absolutely irreducible, or $f=f_1+f_2,$ where both $f_1$ and $f_2$ are absolutely irreducible.
\end{proposition}

\begin{proof}
Set $P:=N(f)=\conv (\mathbf{v}_1,\ldots,\mathbf{v}_r).$ If $P$ is integrally indecomposable, then by Gao irreducibility criterion we have that $f$ is absolutely irreducible. On the other hand, assume that $P$ is not integrally indecomposable. Since $P$ satisfies Goldbach condition, there is an integrally indecomposable polytope $Q:=\conv (\mathbf{w}_1,\ldots,\mathbf{w}_s)$ for some $\mathbf{w}_i\in\mathbb{N}^n$ such that $\conv (\mathbf{v}_1,\ldots,\mathbf{v}_r,\mathbf{w}_1,\ldots,\mathbf{w}_s)$ is integrally indecomposable. Moreover, since
\[
\bigcap_{i=1}^s \supp (\mathbf{w}_i)=\emptyset.
\]
we have that both
\[
f_1:=f+\sum_{i=1}^s \mathbf{x}^{\mathbf{w}_i},\quad f_2:=-\sum_{i=1}^s \mathbf{x}^{\mathbf{w}_i}
\]
are both non divisible by any of the $x_i$'s. Therefore, again by \nameref{Gao irreducibility criterion} we can guarantee that both $f_1$ and $f_2$ are absolutely irreducible and, by construction, $f=f_1+f_2,$ just what we finally wanted to prove.
\end{proof}

So, in our next result we will prove a weaker form of GC for $\mathcal{O}(\mathbb{A}_K^n)$, where $K$ denotes an arbitrary field, and where the number of irreducible polynomials in the decomposition depends on the number of (non-zero) terms of the particular polynomial $H$, but with the advantage that we construct in a more explicit manner the corresponding irreducible polynomials. The only slightly similar explicit result known by the authors in this direction is \cite[Corollary 1.5]{bodin2020schinzel}. However, whereas the proof of \cite[Corollary 1.5]{bodin2020schinzel} is not constructive, our proof is completely explicit and algorithmic, as we shall see soon.

\begin{theorem}\thlabel{teo1}
Let $K$ be any field. Let $\mathcal{O}(\mathbb{A}_K^n)=K[x_1,\ldots,x_n],$ with $n\geq 2$. Then, any polynomial $H$ in $\mathcal{O}(\mathbb{A}_K^n)$ with $r$ non-zero terms can be explicitly written as the sum of at most $2r$ absolutely irreducible polynomials. The zero polynomial can be written as the sum of two absolutely irreducibles.
\end{theorem}

\begin{proof}
We will distinguish essentially three cases among the monomial terms of $H$, and in each of the cases, we will explicitly generate the corresponding absolutely irreducible polynomials that will be used in the decomposition, where there will be a relatively wide spectrum of possibilities for choosing them each time.

Before giving the precise construction, we first explain how we plan to proceed.

\begin{itemize}

\item Starting from $\mathbf{i}\in\supp (f),$ let $H_i=a_{\mathbf{i}}\mathbf{x^{\mathbf{i}}}$ be a monomial of $H.$ We consider the origin $\mathbf{0}=(0,\ldots,0)\in\mathbb{R}^n.$

\item If $\gcd (\mathbf{i})=1,$ then the situation is quite easy to handle; indeed, because of \thref{integrally indecomposable pyramid} the line segment $\overline{\mathbf{0}\mathbf{i}}$ is integrally indecomposable and therefore \nameref{Gao irreducibility criterion} ensures that $\mathbf{x}^{\mathbf{i}}-1$ is an absolutely irreducible polynomial, so we have that
\[
a_{\mathbf{i}}\mathbf{x^{\mathbf{i}}}=(a_{\mathbf{i}}\mathbf{x^{\mathbf{i}}}-1)+(1),
\]
where we add the constant polynomial $(+1)$ to the independent term of $H$ (to be handled later as a particular case). So, we change the monomial 
$a_{\mathbf{i}}\mathbf{x^{\mathbf{i}}}$ by the absolutely irreducible monomial $a_{\mathbf{i}}\mathbf{x^{\mathbf{i}}}-1.$
\item Now, assume that $\gcd (\mathbf{i})>1,$ in this case the idea is to choose a point $\mathbf{w}$ and a hyperplane $G$ containing the line $\mathbf{i}\mathbf{w}$ but not containing the origin, such that we can apply \thref{integrally indecomposable pyramid} with the triangle $\triangle \mathbf{0}\mathbf{i}\mathbf{w}.$ 

\end{itemize}

Now, we give the details. First, let us assume that $n\geq 3$. Let $H_{\mathbf{i}}=a_{\mathbf{i}}x^{\mathbf{i}}$ be a monomial of $H$, where $\mathbf{i}=(i_1,\ldots,i_n)\in\mathbb{N}^n,$ and set $d:=\gcd(\mathbf{i}).$ We have to distinguish two cases.

\begin{enumerate}[(i)]

\item Assume that $d>1.$ Up to permutation of the variables we can assume, without loss of generality, that there exists a subindex $s$ with $i_s\neq 0$ such that $3\leq s\leq n$. Now, set
\[
p:=\left(\prod_{j\in\supp (\mathbf{i})}i_j\right)+2,
\]
and $\mathbf{w}:=(p,p+1,w_3,\ldots,w_{s-1},2i_sp,w_{s+1},\ldots,w_n),$ where the $w_t$ can be any natural number for any subindex $t\neq 1,2,s.$ Notice that, by construction, $p\geq 3$. On the other hand, let $G$ be the hyperplane
\[
i_s(1-2p)(x_1-i_1)+(p-i_1)(x_s-i_s)=0
\]
Note that for any (integer) values of $w_t,$ $\mathbf{i},\mathbf{w}\in G,$ but $\mathbf{0}=(0,\ldots,0)\notin G$. Indeed, let us explicitly check that $\mathbf{0}\notin G.$ Notice that $\mathbf{0}\in G$ if and only if
\[
i_s(1-2p)(-i_1)+(p-i_1)(-i_s)=0.
\]
Since $i_s\neq 0,$ this is equivalent to say that
\[
(1-2p)i_1+i_1-p=0\Longleftrightarrow i_1=\frac{p}{2(1-p)}.
\]
However, $i_1\geq 0$ whereas $\frac{p}{2(1-p)}<0,$ hence we reach a contradiction. Summing up, we have shown that $\mathbf{0}=(0,\ldots,0)\notin G$.

Now, let
\[
A_1:=H_i+\mathbf{x}^{\mathbf{w}}+1,\ A_2:=-\mathbf{x}^{\mathbf{w}}-1.
\]
Since the line segment $\overline{\mathbf{i}\mathbf{w}}$ lies inside the hyperplane $G$ and $\mathbf{0}\notin G,$ the triangle $\triangle\mathbf{0}\mathbf{i}\mathbf{w}$ satisfies the hypothesis of \thref{integrally indecomposable pyramid}. Therefore, $\triangle\mathbf{0}\mathbf{i}\mathbf{w}$ is integrally indecomposable if and only if $\gcd(\mathbf{0}-\mathbf{i},\mathbf{0}-\mathbf{w})=\gcd(\mathbf{i},\mathbf{w})=1.$ But this holds because $\gcd(\mathbf{i},\mathbf{w})|\gcd(p,p+1)=1.$ Moreover, by definition $A_1$ is not divisible by any of the $x_i$'s. Thus, by \nameref{Gao irreducibility criterion}, $A_1$ is absolutely irreducible. Similarly, by \cite[Corollary 4.3]{gao2001absolute}, the segment $\overline{\mathbf{0}\mathbf{w}}$ is integrally indecomposable since $\gcd(\mathbf{w}-\mathbf{0})=\gcd(\mathbf{w})=1$. So, again by \nameref{Gao irreducibility criterion} $A_2$ is also absolutely irreducible. Thus, $H_i=A_1+A_2$ can be written as the sum of two absolutely irreducible polynomials.
It is worth noting that we can apply the former argument for any permutation of the variables $x_1,x_2$ and $x_s$, and the original order of the variables should be taken into account for the final decomposition of $H$ as a sum of absolutely irreducible polynomials. 


\item $\mathbf{i}=\mathbf{0}.$ So, $H_i=c$ is a constant of $K.$ Then,  $c=(x_1+c)+(-x_1),$ where $x_1+c$ and $-x_1$ are trivially absolutely irreducible polynomials. We could also choose here the irreducible polynomials $x_1+x_2+c$ and $-x_1-x_2$, for given an explicit decomposition into (absolutely) irreducible polynomials being non monomials.\footnote{In the proof of \thref{localization} we can see more explicitly and in more detail why this fact can be relevant.}

\end{enumerate} 
Clearly the last two former cases applies as well for $n=2.$ In the first case, we only need to differentiate two simple sub-cases:

\begin{enumerate}[(a)]

\item Assume that $i_1\neq 0.$ In this case, $A_1=H_{\mathbf{i}}+x_1^{i_1}x_2^{(i_1+1)^{(i_2+1)}}+1$ and $A_2=-x_1^{i_1}x_2^{(i_1+1)^{(i_2+1)}}-1$ works, since $\gcd(i_1,(i_1+1)^{(i_2+1)})=1$,  $(i_1+1)^{(i_2+1)}\neq i_2,$ and the line connecting the points $(i_1,i_2)$ and $(i_1,(i_1+1)^{(i_2+1)})$ does not meet the origin $(0,0).$ Note that we can use as well any pair of the form $(i_1^{w_1+1},(i_1+1)^{(i_2+1)+w_2})$, for any $w_1,w_2\in \mathbb{N}.$

\item A similar procedure applies when $i_2\neq 0.$

\end{enumerate}
Summing up, for each non-zero term of $H$ we generate either one or two explicit absolutely irreducible polynomials in the decomposition of $H$. Thus, $H$ can be explicitly expressed as the sum of at most $2r$ absolutely irreducible polynomials.
\end{proof}

\begin{remark}
From the former proof we can deduce straightforwardly that there are infinitely many decompositions of a fixed polynomial in \thref{teo1} as the sum of absolutely irreducible polynomials.
\end{remark}
In \thref{teo1}, we show that any polynomial with $r$ non--zero monomials can be expressed as the sum of at most $2r$ absolutely irreducible polynomials. The reader might think that, in general, $2r$ is a very pessimistic upper bound. However, it turns out that sometimes this upper bound is sharp, as the following remark illustrates, among other things.

\begin{remark}
Note that in the generality of \thref{teo1}, the classic stronger form of Goldbach's conjecture (i.e., with two prime polynomial summands), where the degree of the summands is bounded by the degree of the original polynomial is not true. For example, taking inspiration from \cite[\S 1]{bodin2020schinzel}, let $R=\mathbb{F}_2[x]$, and consider $f(x)=x^2+x.$ Then, due to the elementary fact the only irreducible quadratic polynomial in $R$ is $x^2+x+1,$ we see that $f(x)$ cannot be written as the sum of two quadratic or linear irreducible polynomials in $R$. Nonetheless, our constructive proof of \thref{teo1} gives the explicit description of $f(x)=(x^2+x+1)+(x+1)+(x)$ as the explicit sum of three irreducibles.
\end{remark}
The proof of \thref{teo1}, replacing \nameref{Gao irreducibility criterion} by \nameref{Koyuncu irreducibility criterion}, gives also us the following result.

\begin{theorem}\thlabel{teo2}
Let $R$ be a commutative ring, and let $\mathcal{O}(\mathbb{A}_R^n)=R[x_1,\ldots,x_n],$ with $n\geq 2$. Then, any polynomial $H$ in $\mathcal{O}(\mathbb{A}_R^n)$ with $r$ non-zero terms that are non--zero divisors of $R$ can be explicitly written as the sum of at most $2r$ absolutely irreducible polynomials. The zero polynomial can be written as the sum of two absolutely irreducibles.
\end{theorem}
\begin{proof}

The proof is essentially the same as the given in \thref{teo1}, except that independently of the value of $\gcd (\mathbf{i})$, we will always proceed as if $\gcd (\mathbf{i})\geq 2$. In other words, we will always generate the integrally indecomposable triangle by adding and subtracting a suitable monomial with multi-exponent $\mathbf{w}$, getting expressions of the form 
\[
A_1:=H_i+\mathbf{x}^{\mathbf{w}}+1,\ A_2:=-\mathbf{x}^{\mathbf{w}}-1.
\]
Note that in the corresponding construction of the former polynomials the specific value of $\gcd (\mathbf{i})$ is not relevant. We proceed in the former way because we need to avoid adding units several times to the independent term of $H$, because the resulting constant polynomial can be a zero divisor.
\end{proof}
\thref{teo1} also holds if we replace absolutely irreducible by irreducible, $K$ by any integral domain, and $2r$ by $2$. Let us state this fact explicitly, which is an immediate consequence of \cite[Theorem 2]{pollack2011polynomial}.

\begin{proposition}\thlabel{our result using Pollack theorem}
Let $D$ be an integral domain, and $R=D[x_1,\ldots,x_n],$ with $n\geq 2$. Then, any polynomial in $R$ can be written as the sum of two irreducible polynomials.
\end{proposition}

The disadvantage with \thref{our result using Pollack theorem} is that the proof of \cite[Theorem 2]{pollack2011polynomial} does not  provide a direct way to compute the irreducible ones as fast as the proof of \thref{teo1}.

\section{An Algorithmic Implementation for the Decomposition of Polynomials into Explicit (Absolutely) Irreducible Additive Factors}\label{implemented result section}

Now, we want to illustrate the explicit decomposition found in \thref{teo1} through some examples. The unjustified calculations were done with an algorithm, implemented in Macaulay2 \cite{M2}, that computes such decompositions.\footnote{The interested reader can see the explicit implementation of the algorithm in the following link https://github.com/DAJGomezRamirez/GoldbachDecompostionForPolynomialRings} Our first example is borrowed from \cite[page 102]{GaoLauderalgorithm}.

\begin{example}\label{Gao Lauder example}
Let $R=\mathbb{Q}[x,y],$ and let $f=xy+x+y+1.$ In this case, its Newton polygon is the square given by vertices $(0,0),(1,0),(0,1),(1,1).$ As pointed out in \cite[page 102]{GaoLauderalgorithm}, this polygon is not integrally indecomposable, in fact it has four integral summands. First of all, we exhibit here how our method works in Macaulay2, and afterwards we explain the meaning of the different outputs we obtain.
\begin{verbatim}

i1 : load "goldbach_algorithm.m2";

i2 : R=QQ[x,y];

i3 : f=x*y+x+y+1;

i4 : explicitirreducible(f)

The monomials of
 
x*y + x + y + 1
 
are
 
{x*y, x, y, 1}
 
The exponent set of
 
x*y + x + y + 1
 
is
 
| 1 1 |
| 1 0 |
| 0 1 |
| 0 0 |
 
The w points are
 
| 1 4 |
| 1 2 |
| 2 1 |
| 1 0 |
 
The corresponding monomials given by the w points are
 
    4     2   2
{x*y , x*y , x y, x}
 
The corresponding absolutely irreducible polynomials are
 
    4         2       2
{x*y  + 1, x*y  + 1, x y + 1, x}
 
and also
 
    4               2           2
{x*y  + x*y + 1, x*y  + x + 1, x y + y + 1, x + 1}

\end{verbatim}
On the one hand, the first matrix appearing as output is the one whose rows are exactly the vertices of the Newton polygon of $f.$ On the other hand, our second matrix produces an explicit choice of the points $\mathbf{w}$ constructed along \thref{teo1}. Explicitly, we have that
\begin{align*}
& xy=(xy+xy^4+1)+(-xy^4-1),\quad x=(x+xy^2+1)+(-xy^2-1),\\
& y=(y+x^2y+1)+(-x^2y-1),\ 1=(1+x)+(-x).
\end{align*}
\end{example}

\begin{example}
Let us consider an example with coefficients different from one. Let $R=\mathbb{Q}[x,y,z],$ and let $f=5x^2+3y^2-7z^2-5.$ Its Newton polytope is the tetrahedron given by vertices $(0,0,0),(2,0,0),(0,2,0),(0,0,2).$ In this case, we obtain:

\begin{verbatim}

i1 : load "goldbach_algorithm.m2";

i2 : R=QQ[x,y,z];

i3 : g=2*x^2+3*y^2-7*z^2+5;

i4 : explicitirreducible(g)

The monomials of
 
  2     2     2
2x  + 3y  - 7z  + 5
 
are
 
  2   2   2
{x , y , z , 1}
 
The exponent set of
 
  2     2     2
2x  + 3y  - 7z  + 5
 
is
 
| 2 0 0 |
| 0 2 0 |
| 0 0 2 |
| 0 0 0 |
 
The w points are
 
| 4 4 5 |
| 4 4 5 |
| 4 5 4 |
| 1 0 0 |
 
The corresponding monomials given by the w points are
 
  4 4 5   4 4 5   4 5 4
{x y z , x y z , x y z , x}
 
The corresponding absolutely irreducible polynomials are
 
  4 4 5       4 4 5       4 5 4
{x y z  + 1, x y z  + 1, x y z  + 1, x}
 
and also
 
  4 4 5     2       4 4 5     2       4 5 4     2
{x y z  + 2x  + 1, x y z  + 3y  + 1, x y z  - 7z  + 1, x + 5}
\end{verbatim}
\end{example}
The last example we plan to consider is to illustrate a calculation involving the equation of a projective cubic hypersurface in three variables with mixed terms.

\begin{example}\label{some mixed terms}
Let $R=\mathbb{Q}[x,y,z]$, and let $f=x^3+3x^2y-4y^3+6z^3$ In this case, we obtain:
\end{example}

\begin{verbatim}
i1 : load "goldbach_algorithm.m2";

i2 : R=QQ[x,y,z];

i3 : f=x^3+3*x^2*y-4*y^3+6*z^3;

i4 : explicitirreducible(f)

The monomials of
 
 3     2      3     3
x  + 3x y - 4y  + 6z
 
are
 
  3   2    3   3
{x , x y, y , z }
 
The exponent set of
 
 3     2      3     3
x  + 3x y - 4y  + 6z
 
is
 
| 3 0 0 |
| 2 1 0 |
| 0 3 0 |
| 0 0 3 |
 
The w points are
 
| 6 5 6 |
| 4 4 5 |
| 5 6 6 |
| 5 6 6 |
 
The corresponding monomials given by the w points are
 
  6 5 6   4 4 5   5 6 6   5 6 6
{x y z , x y z , x y z , x y z }
 
The corresponding absolutely irreducible polynomials are
 
  6 5 6       4 4 5       5 6 6       5 6 6
{x y z  + 1, x y z  + 1, x y z  + 1, x y z  + 1}
 
and also
 
  6 5 6    3       4 4 5     2        5 6 6     3       5 6 6     3
{x y z  + x  + 1, x y z  + 3x y + 1, x y z  - 4y  + 1, x y z  + 6z  + 1}

\end{verbatim}
\section{Weak Forms of Goldbach Conjecture over the Integers and Its Localizations}\label{section on Goldbach conjecture and localization}

Let us consider our original equation \eqref{seminal2} in the case that $n_2$ can vary freely with $H$. This case can be considered as a weak form of Goldbach's conjecture. Although a proof of this kind of conjecture over polynomial rings in several variables (as in \thref{teo1}) is not difficult, it is not so immediate and trivially straightforward. However, over classic rings of integers, it is a trivial fact. In other words, by adding enough copies of $2$-s and $3$-s, we can express any natural number bigger than one as finite sum of primes. A little more refined way of obtaining this decomposition with more than two primes is by using inductively Bertrand's Postulate stating that for any real $x\geq 1$ there exists a prime number between $x$ and $2x$ (see either \cite{ramanujan1919proof} or \cite{Bertrandpostulateinductiveproof}).

Now, if we want to extend this result to non-trivial localizations of $\mathbb{Z}$, then we need to be very careful, since localizing we used to lose prime numbers at disposition. So, the direct trick of adding copies of $2$-s and $3$-s cannot be used anymore. Even more, we can lose infinitely many prime numbers on particular localizations of $\mathbb{Z}$, so, Bertrand's postulate would not be a strong tool for generating the precise representations into prime additive factors that we could need. 

Let $S$ be a non-trivial multiplicative system of $\mathbb{Z}$, i.e., $S\neq \{1\},\mathbb{Z}^ *$.\footnote{Note that in the first trivial case the localization is equal to $\mathbb{Z}$, and in the second trivial case, the localization is the field of fractions of $\mathbb{Z}$, i.e., $\mathbb{Q}$. Now, by definition, no form of Goldbach's conjecture is true over any field, because fields contain no prime numbers (or irreducibles).} Then, the weak form of Goldbach's conjecture over the ring $S^{-1}\mathbb{Z}$ is equivalent to say that for all (positive) integer $a$, there exist a natural number $m\in\mathbb{N}$, $s,s_1,\ldots,s_m\in S$ and prime numbers $p_1,\ldots,p_m$ not belonging to $S$ such that 
\[
sa=\sum_{i=1}^ms_ip_i.
\]
Note that the last expression seems to be highly more constrained than the original version over the integers. Effectively, in the localization we can lose almost arbitrarily large collections of prime additive generators although we gain certain additional multiplicative freedom in terms of the elements of the multiplicative system $S$, which is caused by the fact that on $S^{-1}\mathbb{Z}$ each element possesses infinitely many associates. 

In conclusion, at first sight there is no global manner of extend the weak form of Goldbach conjecture on $S^{-1}\mathbb{Z}$. Nonetheless, we can ask ourselves how are distributed topologically into the real (or rational) numbers the collection of elements of $S^{-1}\mathbb{Z}$ that can be written as a finite sum of primes. To answer this question we obtain the following initial positive result.

\begin{theorem}\thlabel{density in localization}
Let $S$ be a non-trivial multiplicative system of $\mathbb{Z}$, $\mathbb{P}_S$ be the set of prime numbers not belonging to $S$, and set
\[
G:=\{w=\sum_{i=1}^m \frac{s'_ip_i}{s_i}:s'_1,s_1,\ldots,s'_m,s_m\in S,\ p_1,\ldots,p_m\in \mathbb{P}_S\}\subseteq S^{-1}\mathbb{Z}.
\]
Then, $G$ is a dense subset of $\mathbb{R}$.
\end{theorem}

\begin{proof}
First, note that $S^{-1}\mathbb{Z}$ is a unique factorization domain whose primes (i.e., irreducibles) are the elements of the form $s'p/s$, where $p\in \mathbb{P}_S$ and $s',s\in S$.\footnote{For several proofs of this fact in other context, see, for example, the proof of \thref{localization}.} So, $G$ is exactly the collection of elements of the localization that can be written as a finite sum of primes. In this way, in order to see that $G$ is dense in $\mathbb{R}$, it is enough to show that any open interval of real numbers contains an element of $G$. This is what we plan to prove in what follows.

Indeed, let $(x_0,y_0)\subseteq \mathbb{R}$ be an open real interval, and set $n_0=\min\{s\in S: s>1\}$. Note that $S\cap\mathbb{N}\neq\emptyset$ because $S$ is a non--trivial multiplicative set. Therefore, $S\cap\mathbb{N}$ is a non--empty subset of natural numbers, hence has a minimum by the well order of $\mathbb{N}$. This justifies that $n_0$ is well defined. Now, fix $p\in \mathbb{P}_S$, with $p>0$. Note that we can choose such $p$ because if not, then all the prime numbers of $\mathbb{Z}$ would be in $S$. Thus, by the Fundamental Theorem of Arithmetic $S=\mathbb{Z}^*$, contradicting the non-triviality of $S$. Choose $e\in \mathbb{N}$ such that $p/n_0^e<y_0-x_0.$ Set $n=\min\{k\in \mathbb{N}: kp/n_0^e\geq y_0\}.$ By definition of $n$, it holds $y_0>(n-1)p/n_0^e$. On the other hand, by the former definitions we obtain that $x_0-y_0<-p/n_0^e$ and $y_0\leq np/n_0^e.$ Thus, adding the former inequalities we get 
\[
x_0=y_0+(x_0-y_0)<\frac{np}{n_0^e}-\frac{p}{n_0^e}=\frac{(n-1)p}{n_0^e}.
\]
Summing up, we have
\[
x_0<\frac{(n-1)p}{n_0^e}<y_0.
\]
Moreover, as observed before, $p/n_0^e$ is a prime in $S^{-1}\mathbb{Z}$, and
\[
\frac{(n-1)p}{n_0^e}=\sum_{i=1}^{n-1}\frac{p}{n_0^e}\in G.
\]
Summing up, we have shown that the interval $(x_0,y_0)$ contains an element of $G$. This finally shows that $G$ is dense in $\mathbb{R}$, just what we finally wanted to show.
\end{proof}

From \thref{density in localization} we can derive a natural and much more interesting result involving series of primes, within the real numbers, into localizations of the integers. The proof is very much like the one to justify that any real number can be expressed as a convergent series of rational ones, and therefore it is left to the interested reader.

\begin{theorem}\thlabel{series of localized primes}
Let $S$ be a non-trivial multiplicative system of $\mathbb{Z}$, and let $\mathbb{P}_S$ be the set of prime numbers not belonging to $S$. Then, any element $v\in S^{-1}\mathbb{Z}$ (resp. any element $v\in \mathbb{R}$) can be written as a convergent series of primes in $S^{-1}\mathbb{Z}$. In other words, there exists $s'_1,s_1,\ldots,s'_n,s_n,\ldots\in S$ and $p_1,\ldots,p_n,\ldots\in \mathbb{P}_S$ such that 
\[
v=\sum_{i=1}^{\infty}\frac{s'_ip_i}{s_i}.
\]
\end{theorem}


\begin{remark}
Notice that \thref{series of localized primes} holds even when $S^{-1}\mathbb{Z}$ possesses only a prime number. Effectively, in such a case the corresponding series would contain suitable forms of (infinitely many) associates of the single prime element. In fact, in the particular case that $S$ consists exactly of the powers of a prime number $p$, we can obtain a more concrete and explicit form of the summands given in \thref{series of localized primes}. Let us state precisely what we mean in the following corollary.
\end{remark}

\begin{corollary}
Let $x$ be a real number. Fix a prime number $q\in \mathbb{N}$. Then, there exists an integer number $m\in \mathbb{Z}$, a collection of prime numbers $\{p_i\}_{i\geq 0}$, all different from $q$, and a increasing collection of integers $\{n_i\}_{i\geq m}$ such that 
\[
x=\sum_{i\geq m}^{\infty}\frac{p_i}{q^{n_i}}.
\]
\end{corollary}

\begin{proof}
Without loss of generality, we can assume that $x\geq 0$ (otherwise, we consider $-x$ and multiply the final (series) representation by minus). By \thref{series of localized primes}, we obtain a series representation with summands of the form $s'_ip_i/s_i$, where $s'$ and $s_i$ are powers of $q$, and $p_i$ is a prime number different from $q$. Moreover, by the proof of \thref{series of localized primes} we can assume that each one of these summands is positive. 

Thus, by simplifying each term and by rearranging all the terms in the series, we can rewrite the series such that the negatives of the final exponents of $q$ generate an increasing sequence. This does not affect the result because any rearrangement of an absolutely convergent series is also convergent, and has the same limit (see, for instance \cite[Chapter 3,Theorem 3.55]{RudinMathematicalAnalysis}). Thus, by considering the permuted series, we obtain our desired representation. 
\end{proof}

Now, we plan to formulate a different variant of Goldbach conjecture, inspired by the results obtained along this section.

\begin{definition}\label{r form of Goldbach conjecture}
Let $r\in \mathbb{N}$. We say that a commutative ring with unity $R$ satisfies the \textit{$r$--form of Goldbach's conjecture} if any element of the ring can be written as a sum of $r$ irreducible elements of $R$.
\end{definition}

In the following result we prove that in order to verify the $r-$Form of Goldbach's conjecture in non-trivial localizations of the integers, it is enough to verify this property on arbitrarily small intervals containing the zero, or arbitrarily big interval of elements bigger than any fix parameter.

\begin{proposition}\thlabel{sufficient condition}
Let $S$ be a non-trivial multiplicative system of $\mathbb{Z}$, let $\mathbb{P}_S$ be the set of prime numbers not belonging to $S$. Fix $\varepsilon,n>0,$ where $n\in\mathbb{N}$ and $\varepsilon\in\mathbb{R}$. Assume that one of the following two conditions hold.

\begin{enumerate}[(i)]

\item The $r-$Form of Goldbach's conjecture holds for any $w\in (-\varepsilon,\varepsilon)\cap S^{-1}\mathbb{Z}$.

\item The $r-$Form of Goldbach's conjecture holds for any $w\in (n,\infty)\cap S^{-1}\mathbb{Z}$.

\end{enumerate}
Then, the $r$--form of Goldbach's conjecture holds for $S^{-1}\mathbb{Z}$.
\end{proposition}

\begin{proof}
Let $v\in S^{-1}\mathbb{Z}$ and $s\in S,$ with $s>1$. Choose $m\in\mathbb{N}$ such that $v/s^m\in(-\varepsilon,\varepsilon)$ (resp. such that $s^mv\in (n,\infty)$). Then, by hypothesis there exist $s'_1,s_1,\ldots,s'_r,s_r\in S$ and $p_1,\ldots,p_r\in \mathbb{P}_S$ such that $v/s^m=\sum_{i=1}^{r}s'_ip_i/s_i$ (resp. $s^mv=\sum_{i=1}^{r}s'_ip_i/s_i$). Then, $v=\sum_{i=1}^{r}s^ms'_ip_i/s_i$ (resp. $v=\sum_{i=1}^{r}s'_ip_i/(s^ms_i)$), giving us a representation of $v$ as the sum of $r$ irreducibles. This finishes our proof.
\end{proof}

\section{Goldbach's Conjecture over Special Classes of Localized Rings}\label{section of more localization}

Our next goal will be to extend \thref{teo1} to suitable localization of coordinate rings of affine varieties over unique factorization domains. With that purpose in mind, we present the following construction.


\begin{construction}\thlabel{making irreducibility in a localization}
Let $K$ be any field, let $R=K[x_1,\ldots, x_n]$ with $n\geq 2.$ For any $\mathbf{i}\in\mathbb{N}^n$ with $\gcd (\mathbf{i})>1$, set
\begin{align*}
& S_{\mathbf{i}}:=\{q(x_1,\ldots,x_n)=\mathbf{x}^{\mathbf{i}}+\mathbf{x}^{\mathbf{w}}+1:\ \mathbf{w}\in\mathbb{N}^n,\ q\text{ is absolutely irreducible}\},\\
& T_{\mathbf{i}}:=\{q(x_1,\ldots,x_n)=\mathbf{x}^{\mathbf{w}}+1:\ \mathbf{w}\in\mathbb{N}^n,\ q\text{ is absolutely irreducible}\},\\
&  U_{\mathbf{i}}:=\{x_1+x_2+c:\ c\in K\},\ W_{\mathbf{i}}:=S_{\mathbf{i}}\cup T_{\mathbf{i}}\cup U_{\mathbf{i}}.
\end{align*}
Moreover, we also set
\[
W:=\bigcup_{\substack{\mathbf{i}\in\mathbb{N}^n\\ \gcd (\mathbf{i})>1}} W_{\mathbf{i}}.
\]
Note that the set $W$ fulfills the condition that it contains all the (absolutely) irreducible polynomials that acts as the additive factors in the proof(s) of \thref{teo1} and \thref{teo2}. In other words, most of the polynomials of $W$ are the ones allowing us to guarantee the weak form of Goldbach condition for arbitrary polynomials in the corresponding polynomial rings. Nonetheless, \nameref{Gao irreducibility criterion} and related results in \cite{gao2001absolute} and \cite{Koyuncuirreducibilityoverrings} give us freedom to choose systematically and explicitly different families of (absolutely) irreducible polynomial as additional classes of additive factors in alternative proofs of these theorems. These facts motivate the following definition.
\end{construction}
\begin{definition}
Let $D$ be an integral domain,  $R=D[x_1,\ldots,x_n]$, and $W$ a collection of (absolutely) irreducible polynomials in $R$. We say that $R$ is a \textbf{system of explicit irreducible polynomials}, if $W$ can be chosen explicitly as an alternative collection of (absolutely) irreducible polynomials in the proof(s) of \thref{teo1} and \thref{teo2}.
\end{definition}
It is straightforward to see that \thref{making irreducibility in a localization} gives an explicit example of a system of explicit irreducible polynomials.
Our next step will be to extend \thref{teo1} and \thref{teo2} in a explicit manner for special kinds of localizations of rings of polynomials over a suitable ring.
The utility of the former definition lies in the fact that it gives more flexibility and a wider range of possibilities to the collection of the multiplicative systems that we can use in such a generalization given by the following theorem.

\begin{theorem}\thlabel{localization}
Let $R$ be a unique factorization domain (UFD), and let $\mathcal{O}(\mathbb{A}_R^n)=R[x_1,\ldots,x_n],$ with $n\geq 2$. Let $W$ be a collection of explicit irreducible polynomial in $\mathcal{O}(\mathbb{A}_R^n)$. Let $S\subsetneq R$ be a multiplicative system generated (multiplicatively) by a set of irreducible polynomials $S_0$ such that $S_0\cap W=\emptyset$. Let $T=S^{-1}R$ be the localization of $R$ with respect to $S$. Then, any element $L$ in $T$ with $r$ non-zero terms can be explicitly written as the sum of at most $2r$ irreducibles.
\end{theorem}

\begin{proof}
First, let us give three simple proofs of a central elementary fact in this context: if $H$ is an irreducible polynomial of $\mathcal{O}(\mathbb{A}_R^n)$ such that $H\notin S$, then $H/1$ is an irreducible element of $T$. 

Sub-proof 1: Effectively, let us assume by the sake of contradiction that $H/1$ is a reducible element of $T$. So, there exists irreducible elements $I_1,\cdots, I_m$ of $\mathcal{O}(\mathbb{A}_R^n)$, not belonging to $S$, with $m\geq 2$, and a $Z\in S$ such that 
\[
ZH=I_1\cdots I_m.
\]
Let us write $Z$ as a product of irreducible elements of $S$, $Z_1,\cdots,Z_q\in S$. Therefore, the last equation becomes
\[
Z_1\cdots Z_qH=I_1\cdots I_m.
\]
Due to the elementary fact that $R[x_1,\cdots, x_n]$ is a unique factorization domain as well (since $R$ so is), and $H$ is irreducible; up to associates (or units), we can cancel $H$ with one $I_j$, for some $j\in [m].$ Let us assume without lost of generality that $j=1$. So, we derive the equation
\[
Z_1\cdots Z_q=I_2\cdots I_m,
\]
which contradicts the unique factorization in $\mathcal{O}(\mathbb{A}_R^n)=R[x_1,\ldots,x_n]$, due to the fact that $I_2$ is an irreducible polynomial do not belonging to $S$, but all the $Z_d$ belong to $S$. An immediate consequence of the former proof is that if $A\in S,$ then the element $H/S$ is irreducible in $T.$

Sub-proof 2: First of all, since $R$ is a UFD, we have that $R[x_1,\ldots,x_n]$ is a UFD. Since the localization of a UFD is also a UFD, we have that $T=S^{-1}R$ is a UFD.

Now, let $H\in R[x_1,\ldots,x_n]$ be an irreducible polynomial such that $H\notin S$. We claim that $H/1$ is irreducible in $T$. Indeed, since $H$ is irreducible and $R[x_1,\ldots, x_n]$ is a UFD, we have that $(H)$ is a prime ideal of $R[x_1,\ldots,x_n].$ Moreover, since $H\notin S$ we also have that $(H/1)$ is a prime ideal of $T.$ In this way, since $(H/1)$ is a prime ideal of $T,$ and $T$ is an integral domain, we have (see for instance \cite[Proposition 4.1.5]{TauvelYu}) that $H/1$ is irreducible in $T.$ In particular, we have that if $A\in S$, then the element $H/A$ is irreducible in $T.$

Sub-proof 3: Since $R$ is a UFD, we have that $R[x_1,\ldots,x_n]$ is a UFD. Since the localization of a UFD is also a UFD, we have that $T=S^{-1}R$ is a UFD.

Now, let $H\in R[x_1,\ldots,x_n]$ be an irreducible polynomial such that $H\notin S$. We claim that $H/1$ is irreducible in $T$. Indeed, $S$ is a multiplicative set generated by irreducible polynomials. Since $R[x_1,\ldots,x_n]$ is a UFD, any irreducible is prime and therefore we have that $S$ is generated by prime elements. Therefore, using \cite[Proposition 1.9]{factorizationintegraldomains} we have that the extension $R[x_1,\ldots, x_n]\subset T$ is inert in the sense of Cohn (see for instance \cite[page 80]{factorizationintegraldomains}). Moreover, we also have that the units of $T$ that belongs to $R[x_1,\ldots,x_n]$ are exactly the units of $R[x_1,\ldots, x_n].$ Therefore, using \cite[Lemma 1.1]{factorizationintegraldomains} we have that $H$ is irreducible in $R[x_1,\ldots,x_n]$ if and only if $H/1$ is irreducible in $T$, proving our claim. In particular, we have that if $A\in S$, then the element $H/A$ is irreducible in $T.$

Finally, let us consider a rational polynomial in $T$, let us say $H/W$, where $W\in \mathcal{O}(\mathbb{A}_R^n)$ and $W\in S.$ By \thref{teo2}, $H$ can be written as the sum of at most $2r$ absolutely irreducible polynomials, i.e.,
\[
H=\sum_{b=1}^pI_b,
\]
for absolutely irreducible polynomials $I_1,\cdots, I_p;$ with $p\leq 2r.$ Moreover, by construction of $S$, $I_b\notin S,$ for all $b\in [p].$ Thus, by the former central elementary fact, 
\[
H/W=\sum_{b=1}^p(I_b/W),
\]
where the elements $I_b/W$ are irreducible in $T$ for all $b\in [p]$. This finishes our constructive proof. 
\end{proof}

\begin{remark}
Notice that, since $R$ is an integral domain, in the proof of \thref{localization} we can choose the explicit irreducible elements described in the proof of \thref{teo1} if we want.
\end{remark}

A particular special case of \thref{localization} is given by the next statement. 

\begin{corollary}
Let $R$ be a unique factorization domain, and let $\mathcal{O}(\mathbb{A}_R^n)=R[x_1,\ldots,x_n],$ with $n\geq 2$. Let $S_m\subsetneq R$ be a multiplicative system consisting of the monomials of $R[x_1,\ldots,x_n]$ . Let $T=S_m^{-1}R$ the localization of $R$ in $S_m$. Then, any element $L$ in $T$ with $r$ non-zero terms can be explicitly written as the sum of at most $2r$ irreducibles. The zero element can be written as the sum of two  irreducibles.
\end{corollary}

\begin{proof}
In \thref{localization} we can set $S=S_m$ due to the fact that all the irreducible polynomials in the \thref{making irreducibility in a localization} are non monomials.
\end{proof}

\begin{remark}\label{non straightforward extension of teo1}
One natural way to proceed to generalize \thref{teo1} for coordinate rings of varieties would be to consider highly simple coordinate rings of varieties given by irreducible polynomials produced by \nameref{Gao irreducibility criterion}, and, subsequently try to use the same method of the proof of \thref{teo1} for constructing explicitly the (absolutely) irreducibles polynomials for each monomial of the polynomial in consideration. 

However, this methodological line of generalization seems not to work so straightforwardly due to the fact that in a quotient ring the property of being (absolutely) irreducible could be materialized qualitatively in a strong different manner as in the ring of polynomials. Let us show exactly what we mean with the following example.

Consider the (absolutely) irreducible polynomial $g=w^px^{p+1}+y^{2pi}\in R:=K[w,x,y]$, for some positive integers $p,i\in\mathbb{N}$, due to \nameref{Gao irreducibility criterion}. Set $V=V(g)$. Then, in $\mathcal{O}(V)=R/(g)$ a typical polynomial that we considered in the constructive proof of \thref{teo1} has the form of 
\[
A=w^px^{p+1}y^{2pi}+1.
\]
Now, in $R$, $A$ is an (absolutely) irreducible polynomial by \nameref{Gao irreducibility criterion}. Nonetheless, in $\mathcal{O}(V)$, $A$ turns out to be a reducible polynomial. Indeed, regard both polynomials $g$ and $A$ as polynomials in the variable $y$; that is, $g,\ A\in K[x,w][y]$. If we perform here the euclidean division of $A$ by $g$ we obtain
\[
A=(w^px^{p+1})g+(1-w^{2p}x^{2(p+1)}),
\]
and therefore, keeping in mind this equality we clearly have that
\[
A\equiv 1-w^{2p}x^{2(p+1)}\equiv (1-w^px^{p+1})(1+w^p x^{p+1})\pmod g.
\]
In this way, if $W,X,Y$ denote the classes in $\mathcal{O}(V)$ of the corresponding variables in $R$, then the following equation holds in $\mathcal{O}(V)$.
\[
W^pX^{p+1}Y^{2pi}+1=(W^pX^{p+1}+1)(Y^{2pi}+1).
\]
This equation, taking into account that $y^{2pi}\equiv -w^p x^{p+1}\pmod g,$ can also be written in the following way.
\[
W^pX^{p+1}Y^{2pi}+1=(1-Y^{2pi})(1+Y^{2pi})=(1-Y^{pi})(1+Y^{pi})(1+Y^{2pi}).
\]
Finally, using that $1-Y^{pi}=(1-Y)(Y^{pi-1}+Y^{pi-2}+\ldots+Y+1),$ we end up with the following factorization of our polynomial $A$ in $R/(g).$
\[
W^pX^{p+1}Y^{2pi}+1=(1-Y)(Y^{pi-1}+Y^{pi-2}+\ldots+Y+1)(1+Y^{pi})(1+Y^{2pi}).
\]
In conclusion, in $\mathcal{O}(V)$ some of the fundamental (absolutely) irreducible polynomials used for our additive decomposition of polynomials are not (absolutely) irreducible anymore. 
\end{remark}

\section{Characterizing Goldbach's Conjecture over suitable Forms of Forcing Algebras}\label{section on forcing algebras}

Let us continue with an elementary result involving a characterization of GC in the special case of forcing algebras over an algebraically closed field $K$. Forcing algebras emerge naturally when we want to translate within an algebraic canonical structure \emph{how close} is an element $f$ of a commutative ring with unity $R$ of belonging to a finitely generated ideal $I=(f_1,\ldots,f_m)\subseteq R$ \cite[Section 4]{solidclosure}.

\begin{proposition}\thlabel{Goldbach conjecture and forcing algebras}
Let $K$ be any field, let $R=K[x_1,\ldots,x_n]$ the ring of polynomials over $K$ in finitely many variables; $f,f_1,\ldots,f_n\in K$, with  $f_i\neq 0$, for some index $i$; and 
\[
A=K[x_1, \ldots, x_n] /(f_1x_1 + \ldots + f_nx_n+f ) 
\]
the corresponding forcing algebra. Then, the following statements hold.

\begin{enumerate}[(i)]

\item Assume that $n\geq 3.$ Then, any element of $A$ can be written as the sum of two irreducible elements of $A.$

\item If, in addition, $K$ is algebraically closed, then any element of $A$ can be written as the sum of two irreducible elements if and only if $n\geq 3.$ 

\end{enumerate}
\end{proposition}

\begin{proof}
First of all, we prove part (i). Indeed, assume that $n\geq 3.$ Then, we have a $K$--algebra isomorphism
\[
A\cong K[x_1,\ldots,x_{i-1},x_{i+1},x_{i+2},\ldots,x_n]
\]
given by sending $x_i\mapsto -\sum_{j\neq i}(f_j/f_i)x_j-(f/f_i)$ and $x_r\mapsto x_r$ for $r\neq i$. So, $A$ is isomorphic to the ring of polynomials in $(n-1)\geq 2$ variables. Thus, by \cite[Theorem 2]{pollack2011polynomial} any non-constant polynomial of $A$ of degree $m\geq 1$ can be written as the sum of two irreducible polynomials of degree $m$. Now, if $h$ is a constant polynomial in $A$, then $h$ can be written as the sum of two irreducible polynomials $h=x_j+(-x_j+h)$, where $j\neq i$, and $x_j$ and $-x_j+h$ are obviously irreducible polynomials in $A.$ This completes the proof of part (i)

In order to prove (ii), hereafter we assume that $K$ is algebraically closed. If $n\geq 3$ then we are done by part (i), so we suppose that $n \leq 2.$ In the case that $n=1,$ by the same argument as before, $A$ is isomorphic to $K.$ Therefore, $A$ do not fulfill our thesis, since all the elements of $A$, but zero, are units, which, by definition are not irreducible. When $n=2$, $A\cong K[T]$. Since $K$ is algebraically closed, the only irreducible polynomials of $A$ have degree one. Then, no polynomial in $A$ of degree $\geq 2$ can be written as the sum of two irreducible polynomials. 
\end{proof}

\section{The Strong Goldbach Condition on Special Classes of Coordinate Rings of Affine Varieties over Several Types of Fields}\label{section on strong Golbach condition}

Our next proposition involves a certain class of affine varieties whose rings of regular functions (or coordinate rings) fulfills the stronger version of Goldbach's conjecture in \thref{Goldbach conjecture and forcing algebras}, i.e. any regular function can be written as the sum of two irreducible regular functions. Let us formulate this property in a precise manner before stating our result.

\begin{definition}\label{strong Goldbach conjecture: definition}
Let $R$ be a commutative ring. We say that $R$ satisfies the \textbf{Strong Goldbach Condition (SGC)} if any element of $R$ can be written as the sum of two irreducible elements of $R.$
\end{definition} 

\begin{proposition}\thlabel{prop2}
Let $K$ be any field, let $f_1,\ldots,f_r\in K[x_1,\ldots,x_n]$ be such that $X_n=V(f_1,\ldots,f_r)\subsetneq \mathbb{A}_K^n$ is an irreducible affine variety with infinitely many points. Let $X_{n+1}=V(f_1,\ldots,f_r)\subseteq \mathbb{A}_K^{n+1}$ be the corresponding irreducible variety embedded in $\mathbb{A}_K^{n+1}$. Then, the ring of regular functions of $X_{n+1}$, $\mathcal{O}(X_{n+1})$ fulfills the SGC. 
\end{proposition}

\begin{proof}
It is straightforward to see that  
\begin{align*}
\mathcal{O}(X_{n+1})& \cong K[x_1,\ldots,x_{n+1}]/I(X_{n+1})\cong K[x_1,\ldots,x_{n+1}]/(f_1,\ldots,f_r)\\
& \cong K[x_1,\ldots,x_n]/(f_1,\ldots,f_r))[x_{n+1}]=\mathcal{O}(X_n)[x_{n+1}].
\end{align*}
Moreover, by hypothesis and by Hilbert Basis Theorem, the ring $\mathcal{O}(X_n)$ is a Noetherian integral domain with infinitely many maximal ideals (corresponding to the infinitely many points of $X_n$). Thus, the ring $\mathcal{O}(X_n)$ fulfills the hypothesis of \cite[Theorem 1]{pollack2011polynomial}. Therefore, any polynomial of degree $m\geq 1$ in $\mathcal{O}(X_n)[x_{n+1}]$ can be expressed as the sum of two irreducibles of degree $m$. With the same argument, we also have that any constant polynomial can be written as the sum of two irreducible polynomials. In conclusion, the ring $\mathcal{O}(X_{n+1})$ satisfies the SGC.
\end{proof}

Our next corollary involves an equivalent condition in terms of the dimension of the corresponding affine variety.

\begin{corollary}\thlabel{dimension zero}
Let $K$ be a field, $X_n=V(f_1,\ldots,f_r)\subsetneq \mathbb{A}_K^n$ an irreducible affine variety for some polynomials $f_1,\ldots,f_r\in K[x_1,\ldots,x_n].$ Let $X_{n+1}=V(f_1,\ldots,f_r)\subseteq\mathbb{A}_K^{n+1}$ be the corresponding irreducible variety embedded in $\mathbb{A}_K^{n+1}.$ Assume that $\dim (X_n)>0$. Then, $\mathcal{O}(X_{n+1})$ fulfills the SGC. 
\end{corollary}
\begin{proof}
Due to \cite[Chapter 9, Section 4, Proposition 6]{CoxLittleOShea} we have that $\dim(X_n)>0$ if and only if the affine variety $X_n$ has infinitely many points. So, our corollary follows immediately from \thref{prop2}.
\end{proof}

In the case of algebraic varieties over an algebraically closed field, the fact that an irreducible variety has infinitely many points can be formulated in several equivalent ways. We want to single out the following one in the next statement.

\begin{corollary}\thlabel{the case of an algebraically closed field}
Let $K$ be an algebraically closed field, $X_n=V(f_1,\ldots,f_r)\subsetneq \mathbb{A}_K^n$ an irreducible affine variety for some polynomials $f_1,\ldots,f_r\in K[x_1,\ldots,x_n].$ Let $X_{n+1}=V(f_1,\ldots,f_r)\subseteq\mathbb{A}_K^{n+1}$ be the corresponding irreducible variety embedded in $\mathbb{A}_K^{n+1}.$ Assume that $\mathcal{O}(X_n)$ is an infinite dimensional $K-$vector space. Then, $\mathcal{O}(X_{n+1})$ fulfills the SGC. 
\end{corollary}
\begin{proof}
First of all, since $K$ is algebraically closed, by the Finiteness Theorem \cite[Chapter 5, Section 3, Theorem 6]{CoxLittleOShea} we have that $\mathcal{O}(X_n)$ is an infinite dimensional $K$--vector space if and only if the affine variety $X_n$ has infinitely many points. So, our corollary follows immediately from \thref{prop2}.
\end{proof}

\begin{remark}
Note that in \thref{the case of an algebraically closed field}, by the Finiteness Theorem \cite[Chapter 5, Section 3, Theorem 6]{CoxLittleOShea} (see also \cite[Proposition 3.7.1]{RobbianoKreuzerCA1}), we can replace the condition involving the endlessness of the dimension of $\mathcal{O}(X_n)$ as $K$--vector space by any other of the conditions described there, which describe properties involving certain classes of monomials (not) belonging to the leading terms of the elements of the polynomials in $I=(f_1,\ldots,f_r)\subseteq \mathcal{O}(X_n)$, as well as the corresponding Gr{\"o}bner basis.
\end{remark}

The reader will easily note that, in order to apply \thref{prop2} we need to guarantee that our affine variety is positive dimensional and irreducible. Irreducibility is in general a property that it is not so easy to check, however we want to single out the case of plane algebraic curves defined over the reals, where irreducibility can be characterized in a nice way.

\begin{corollary}\thlabel{the case of real plane algebraic curves}
Let $f\in\mathbb{R}[x,y]$ be an irreducible indefinite polynomial, let $X:=V(f)\subseteq\mathbb{A}_{\mathbb{R}}^2$ be the corresponding affine real plane algebraic curve, and let $Y:=V(f)\subseteq\mathbb{A}_{\mathbb{R}}^3$ be the corresponding variety embedded in $\mathbb{A}_{\mathbb{R}}^3$. Then, the ring of regular functions of $Y$, $\mathcal{O}(Y)$ fulfills the SGC. 
\end{corollary}

\begin{proof}
On the one hand, since $f$ is indefinite and irreducible, $X$ is irreducible by \cite[Theorem 15]{delaPuentenotes}. On the other hand, \cite[Corollary 9]{delaPuentenotes} implies that $X$ has infinitely many points. In this way, the result follows immediately again from \thref{prop2}.
\end{proof}

\begin{remark}\label{in the plane we can not go further}
It is known \cite[Algebraic Lemma 4]{delaPuentenotes} that, given any field $K$ and given two polynomials $f,\ g\in K[x,y]$ both of positive degree and coprime, then we have that $V(f)\cap V(g)$ is either empty or finite. Therefore, in this case, we can not apply \thref{prop2}.
\end{remark}

\section{Some enlightening examples in the one dimensional case}\label{section on the curve case}



As we have already illustrated in \thref{the case of real plane algebraic curves}, for affine curves the situation concerning Goldbach's condition is more delicate. We want to illustrate this fact with a pair of concrete examples.

\begin{example}
Let $n\in \mathbb{N}$, $n\geq 1$, let $K$ be an algebraically closed field, and let $f=x^ny-1\in K[x,y]$, $X=V(f)$. Then, $\dim(X)=\dim(\mathcal{O}(X))=1$, and, since the class of $x^n$ is invertible in $\mathcal{O}(X)$, we immediately check that 
\[
\mathcal{O}(X)\cong K[x,x^{-1},y]/(y-x^{-n})\cong K[x,x^{-1}].
\]
Now, since $K$ is algebraically closed, any Laurent polynomial $h\in K[x,x^{-1}]$ such that the difference between the degree of $h$ and the valuation of $h$ is bigger strictly than one, then $h$ can be factored as $h=x^zh'$, where $z\in\mathbb{Z}$, and $h'\in K[x],$ with $\deg(h')\geq 2$. So, this kind of polynomials are not irreducible because, since $K$ is algebraically closed, the only irreducible polynomials are the ones of degree one. Thus, we deduce that the irreducible elements in $K[x,x^{-1}]$ are of the form $x^w(a_1x+a_0)$, where $w\in \mathbb{Z}$ and $a_1,a_0\in K$. From the former fact, we can conclude that any Laurent polynomial with at least five terms cannot be written as the sum of two irreducible polynomials. So, $\mathcal{O}(X)$ do not fulfill the SGC.
\end{example}

\begin{example}
Let $K$ be any field, let $f=x^3y^2-1\in K[x,y]$, $X=V(f)$. Again, $\dim(X)=\dim(\mathcal{O}(X))=1$, and, because $x^3$ is invertible in $\mathcal{O}(X)$, one sees that 
\[
\mathcal{O}(X)\cong K[x,x^{-1},y]/(y^2-x^{-3}).
\]
So, each element $G\in\mathcal{O}(X)$ can be written essentially in the form
\[
G=H_1(X,X^{-1})Y+H_0(X,X^{-1}),
\] where the capital letters denote the classes of the corresponding polynomials and variables $x,x^{-1},y,h_0,h_1,g\in K[x,x^{-1},y].$ So, $G$ can be written as follows 
\[
G=[(H_1(X,X^{-1})-1)Y-1]-[Y+(H_0(X,X^{-1})+1)].
\]
Now, one can easily check that both polynomials (classes) $[(h_1(x,x^{-1})-1)y-1]$ and $[y+(h_0(x,x^{-1})+1)]$ are irreducible in $\mathcal{O}(X),$ due to the fact that each pair of coefficients in $K[X,X^{-1}]$ of both of them are coprime in $K[x,x^{-1}]$. In conclusion, $\mathcal{O}(X)$ satisfies the SGC. 
\end{example}

\section*{Acknowledgements}
Part of this work was done when D. A. J. G\'omez Ram\'irez visited the University of Valladolid in November, 2023. The authors would like to thank Ricardo Garc\'ia and Pedro Gonz\'alez P\'erez for some comments concerning the content of this paper. Alberto F. Boix was partially supported by Spanish Ministerio de Econom\'ia y Competitividad grant PID2019-104844GB-I00. D. A. J. G\'omez Ram\'irez would like to thank Michelle G\'omez for all her support and love.

\bibliographystyle{alpha}
\bibliography{bibliography}

\begin{thebibliography}{EHM05}

\bibitem[AAZ92]{factorizationintegraldomains}
D.~D. Anderson, D.~F. Anderson, and M.~Zafrullah.
\newblock Factorization in integral domains. {II}.
\newblock {\em J. Algebra}, 152(1):78--93, 1992.

\bibitem[Ari18]{Bertrandpostulateinductiveproof}
B.~R. Arif.
\newblock An inductive proof of {B}ertrand's postulate.
\newblock {\em Ganit}, 38:85--87, 2018.

\bibitem[BDN20]{bodin2020schinzel}
A.~Bodin, P.~D\`ebes, and S.~Najib.
\newblock The {S}chinzel hypothesis for polynomials.
\newblock {\em Trans. Amer. Math. Soc.}, 373(12):8339--8364, 2020.

\bibitem[BGR16]{gomezbrennernormality}
H.~Brenner and D.~A.~J. G\'{o}mez-Ram\'{\i}rez.
\newblock Normality and related properties of forcing algebras.
\newblock {\em Comm. Algebra}, 44(11):4769--4793, 2016.

\bibitem[CLO15]{CoxLittleOShea}
D.~A. Cox, J.~Little, and D.~O'Shea.
\newblock {\em Ideals, varieties, and algorithms. {A}n introduction to computational algebraic geometry and commutative algebra}.
\newblock Undergraduate Texts in Mathematics. Springer, Cham, fourth edition, 2015.

\bibitem[dlP02]{delaPuentenotes}
M.~J. de~la Puente.
\newblock Real plane algebraic curves.
\newblock {\em Expo. Math.}, 20(4):291--314, 2002.

\bibitem[EHM05]{effinger2005integers}
G.~Effinger, K.~Hicks, and G.~L. Mullen.
\newblock Integers and polynomials: comparing the close cousins {$\mathbb{Z}$} and {$\mathbb{F}_q[x]$}.
\newblock {\em Math. Intelligencer}, 27(2):26--34, 2005.

\bibitem[Eul15a]{GoldbachtoEulerconjecture}
L.~Euler.
\newblock {\em Leonhardi {E}uleri---{O}pera omnia. {S}eries 4 {A}. {C}ommercium epistolicum. {V}ol. 4.1. {L}eonhardi {E}uleri commercium epistolicum cum {C}hristiano {G}oldbach. {P}ars {I}/{C}orrespondence of {L}eonhard {E}uler with {C}hristian {G}oldbach. {P}art {I}}.
\newblock Springer, Basel, 2015.
\newblock Original texts in Latin and German.

\bibitem[Eul15b]{GoldbachtoEulerconjecturetranslated}
L.~Euler.
\newblock {\em Leonhardi {E}uleri---{O}pera omnia. {S}eries 4 {A}. {C}ommercium epistolicum. {V}ol. 4.2. {L}eonhardi {E}uleri commercium epistolicum cum {C}hristiano {G}oldbach. {P}ars {II}/{C}orrespondence of {L}eonhard {E}uler with {C}hristian {G}oldbach. {P}art {II}}.
\newblock Springer, Basel, 2015.
\newblock Original text translated from the Latin and German.

\bibitem[Gao01]{gao2001absolute}
S.~Gao.
\newblock Absolute irreducibility of polynomials via {N}ewton polytopes.
\newblock {\em J. Algebra}, 237(2):501--520, 2001.

\bibitem[GL01]{GaoLauderalgorithm}
S.~Gao and A.~G.~B. Lauder.
\newblock Decomposition of polytopes and polynomials.
\newblock {\em Discrete Comput. Geom.}, 26(1):89--104, 2001.

\bibitem[GR13]{gomez2013homological}
D.~A.~J. G\'omez~Ram\'irez.
\newblock {\em Homological Conjectures, Closure Operations, Vector Bundles and Forcing Algebras}.
\newblock PhD thesis, Universidad Nacional de Colombia with cooperation of the University of Osnabrueck, 2013.

\bibitem[GR20]{AMI}
D.~A.~J. G\'{o}mez~Ram\'{\i}rez.
\newblock {\em Artificial mathematical intelligence---cognitive, (meta)mathematical, physical and philosophical foundations}.
\newblock Springer, Cham, 2020.

\bibitem[GS]{M2}
D.~R. Grayson and M.~E. Stillman.
\newblock Macaulay2, a software system for research in algebraic geometry.
\newblock Available at \url{http://www2.macaulay2.com}.

\bibitem[GW20]{gortz2010algebraic}
U.~G\"{o}rtz and T.~Wedhorn.
\newblock {\em Algebraic geometry {I}. {S}chemes---with examples and exercises}.
\newblock Springer Studium Mathematik---Master. Springer Spektrum, Wiesbaden, second edition, 2020.

\bibitem[Hay65]{hayes1965goldbach}
D.~R. Hayes.
\newblock A {G}oldbach theorem for polynomials with integral coefficients.
\newblock {\em Amer. Math. Monthly}, 72:45--46, 1965.

\bibitem[Hoc94]{solidclosure}
M.~Hochster.
\newblock Solid closure.
\newblock In {\em Commutative algebra: syzygies, multiplicities, and birational algebra ({S}outh {H}adley, {MA}, 1992)}, volume 159 of {\em Contemp. Math.}, pages 103--172. Amer. Math. Soc., Providence, RI, 1994.

\bibitem[Koy13]{Koyuncuirreducibilityoverrings}
F.~Koyuncu.
\newblock Integral polytopes and polynomial factorization.
\newblock {\em Turkish J. Math.}, 37(1):18--26, 2013.

\bibitem[KR08]{RobbianoKreuzerCA1}
M.~Kreuzer and L.~Robbiano.
\newblock {\em Computational commutative algebra 1}.
\newblock Springer-Verlag, Berlin, 2008.
\newblock Corrected reprint of the 2000 original.

\bibitem[Kra96]{kraft1996challenging}
H.~Kraft.
\newblock Challenging problems on affine {$n$}-space.
\newblock In {\em S\'eminaire Bourbaki. Volume 1994/95. Expos\'es 790-804}, number 237, pages 295--317. Paris: Soci{\'e}t{\'e} Math{\'e}matique de France, 1996.

\bibitem[Par20]{Goldbachpowerseries}
E.~Paran.
\newblock Twin-prime and {G}oldbach theorems for {$\mathbb{Z}[\![x]\!]$}.
\newblock {\em J. Number Theory}, 213:453--461, 2020.

\bibitem[Pol11]{pollack2011polynomial}
P.~Pollack.
\newblock On polynomial rings with a {G}oldbach property.
\newblock {\em Amer. Math. Monthly}, 118(1):71--77, 2011.

\bibitem[Ram00]{ramanujan1919proof}
S.~Ramanujan.
\newblock A proof of {B}ertrand's postulate [{J}. {I}ndian {M}ath. {S}oc. {\bf 11} (1919), 181--182].
\newblock In {\em Collected papers of {S}rinivasa {R}amanujan}, pages 208--209. AMS Chelsea Publ., Providence, RI, 2000.

\bibitem[RS98]{RattanStewart}
A.~Rattan and C.~Stewart.
\newblock Goldbach's conjecture for {$\mathbb{Z}[x]$}.
\newblock {\em C. R. Math. Acad. Sci. Soc. R. Can.}, 20(3):83--85, 1998.

\bibitem[Rud76]{RudinMathematicalAnalysis}
W.~Rudin.
\newblock {\em Principles of mathematical analysis}.
\newblock International Series in Pure and Applied Mathematics. McGraw-Hill Book Co., New York-Auckland-D\"{u}sseldorf, third edition, 1976.

\bibitem[Sai06]{Saidakintegerpolynomials}
F.~Saidak.
\newblock On {G}oldbach's conjecture for integer polynomials.
\newblock {\em Amer. Math. Monthly}, 113(6):541--545, 2006.

\bibitem[Stu96]{SturmfelsGBCPbook}
B.~Sturmfels.
\newblock {\em Gr\"{o}bner bases and convex polytopes}, volume~8 of {\em University Lecture Series}.
\newblock American Mathematical Society, Providence, RI, 1996.

\bibitem[TY05]{TauvelYu}
P.~Tauvel and R.~W.~T. Yu.
\newblock {\em Lie algebras and algebraic groups}.
\newblock Springer Monographs in Mathematics. Springer-Verlag, Berlin, 2005.

\bibitem[Vau16]{vaughan2016goldbach}
R.~C. Vaughan.
\newblock Goldbach's conjectures: a historical perspective.
\newblock In {\em Open problems in mathematics}, pages 479--520. Springer, [Cham], 2016.

\end{thebibliography}
\end{document}